\documentclass[12pt,twoside,leqno,final]{amsart}
\usepackage{amsmath}
\usepackage{amssymb}

\theoremstyle{plain}
\newtheorem{thm}{Theorem}[section]
\newtheorem{lem}[thm]{Lemma}
\newtheorem{defn}[thm]{Definition}
\newtheorem{prop}[thm]{Proposition}
\newtheorem{cor}[thm]{Corollary}

\newcommand{\C}{{\mathbb C}}

\newcommand{\B}{{\mathbb B}}
\newcommand{\bB}{{\mathbb S}}

\newcommand{\A}{{\mathcal A}}
\newcommand{\Ra}{{\mathcal R}}
\newcommand{\I}{{\mathcal I}}
\newcommand{\T}{{\mathcal T}}
\newcommand{\K}{{\mathcal K}}
\newcommand{\Q}{{\mathcal Q}}

\newcommand{\M}{{\mathcal M}}
\renewcommand{\L}{\mathcal L}

\newcommand{\e}{\varepsilon}
\newcommand{\p}{\partial}

\newcommand{\z}{\zeta}

\title[Corona theorem in Hardy and Morrey spaces]{The corona theorem in weighted Hardy and Morrey spaces}

\author{Carme Cascante}
\address{C. Cascante: Dept.\ Matem\`atica Aplicada i An\`alisi,Universitat  de Barcelona, Gran Via 585, 08071 Barcelona, Spain}
\email{cascante@ub.edu}
\author{Joan F\`abrega}
\address{Joan F\`abrega: Dept.\ Matem\`atica Aplicada i An\`alisi,Universitat  de Barcelona, Gran Via 585, 08071 Barcelona, Spain}
\email{joan$_{-}$fabrega@ub.edu}
\author{Joaqu\'\i n M. Ortega}
\address{J.M. Ortega: Dept.\ Matem\`atica Aplicada i An\`alisi,Universitat  de Barcelona, Gran Via 585, 08071 Barcelona, Spain}
\email{ortega@ub.edu}
\subjclass[2000]{47,32}

\date{\today}
\thanks{Partially supported by  DGICYT Grant MTM2008-05561-C02-01, DURSI Grant 2009SGR 1303 and Grant MTM2008-02928-E}

\begin{document}

\begin{abstract} The main goal of this paper is to give an unified proof of the corona problem on weighted Hardy spaces and on Morrey spaces. We use a technique that allows to reduce the problem to the Hardy spaces $H^2(\theta)$.
\end{abstract}

\maketitle

\section{Introduction}

Let $\B$ denote the unit ball of $\C^n$ and $\bB$ its boundary. Let $d\nu$ and $d\sigma$ denote the corresponding Lebesgue measures on $\B$ and $\bB$. In \cite{An1} and \cite{Am} the authors introduce the so called $H^p-$corona problem, which states that if $g_1,\ldots,g_m$ are bounded analytic functions on $\B,$ satisfying
\begin{equation}\label{eqn:coronacond}
\inf\{|g(z)|^2=|g_1(z)|^2+\cdots+|g_m(z)|^2:\,z\in \B\}>0,
\end{equation}
then for any $1\le p<\infty,$ the map $\M_g:H^p\times \cdots\times H^p\rightarrow H^p$ defined by $(f_1,\ldots,f_m)\rightarrow g_1f_1+\ldots +g_mf_m$ is surjective.
Analogous problems for Bergman, Lipschitz, Besov, BMOA and Bloch spaces in the unit ball or in a more general class of domains have been considered by several authors (see for instance \cite{An-Car2, An-Car3}, \cite{Cos-Saw-Wick1} \cite{Kr-Li}, \cite{Lin}, \cite{Or-Fa1, Or-Fa2, Or-Fa3} ).

The goal of this paper is the study of the corona problem for holomorphic weighted  Hardy spaces with respect to weights on $\bB$ of the Muckenhoupt class $\A_p$, and for Morrey spaces.

For $1< p<\infty$, and $\theta\in \A_p$, the Hardy space $H^p(\theta)$ consists of holomorphic functions $f$ on $\B$ such that 
\begin{equation}
\|f\|_{H^p(\theta)}=\left(\sup_r\int_{\bB} |f(r\z)|^p\theta(\z) d\sigma(\z)\right)^{1/p}<\infty.
\end{equation}

For $1< p<\infty$ and $-1< s\le n/p,$ we also define the Morrey-Campanato space $M^{p,s}$ on $\bB$ given by 
$$ M^{p,s}=\left\{f\in L^p(\bB):\,\|f\|_{M^{p,s}}<\infty\right\},$$
where 
\begin{equation}\label{eqn:Morreynorm1}
\|f\|_{{p,s}}=\|f\|_p+\sup_{I_{\z,\e}}\left(\e^{sp-n}\int_{I_{\z,\e}}|f(\eta)-f(\z)|^pd\sigma(\eta)\right)^{1/p},
\end{equation}
$\|f\|_p$ denotes the usual $L^p(\bB)-$ norm of $f,$ and 
$I_{\z,\e}=\{\eta\in \bB;\,|1-\bar\z \eta|<\e\}.$

It is clear that for $s=n/p$ the space  $M^{p,n/p}$ coincides with $L^p(\bB)$ and that for $s=0$,  $M^{p,0}$  coincides with the non isotropic BMO space. It is also well-known that  for $-1<s<0$  the space $ M^{p,s}$ coincides with the non isotropic Lipschitz space $\Lambda_s.$

Let  $HM^{p,s}=M^{p,s}\cap H^p$ be the corresponding holomorphic Morrey space. 

The main goal of this paper is to obtain necessary and sufficient conditions on  holomorphic functions $g_1,\ldots, g_m$ on $\B,$ such that  $\M_g$ maps $ X\times\cdots\times X$ onto $X$, where $X=H^p(\theta)$ or $X=HM^{p,s}$.

We believe that the interest of the paper lies not only on the results but on the techniques that allow to reduce the proof for $H^p(\theta)$ to the particular case $p=2$ and any weight in $\A_2$. The general result for every weighted Hardy space $H^p(\theta)$ is then a consequence of Rubio de Francia extrapolation theorem. This method gives an alternative simpler proof, even for the decompositions for the unweighted Hardy spaces $H^p$.

The corona problem for the Morrey-Campanato spaces in the scale $-1< s < 0$, corresponding to  Lipschitz spaces, was  considered in \cite{Kr-Li}, the case $s=0$ corresponding to  BMOA in \cite{Or-Fa1} and \cite{An-Car3}, and the case $s=n/p$, corresponding to the $H^p$ space, has been previously mentioned. 
Therefore,  only remains to consider the case $0<s<n/p.$ 

In this case the  norm \eqref{eqn:Morreynorm1} is equivalent to the Morrey norm (see for instance \cite{Ku-John-Fu})
\begin{equation}\label{eqn:Morreynorm2}
\|f\|_{M^{p,s}}=\sup_{I_{\z,\e}}\left(\e^{sp-n}\int_{I_{\z,\e}}|f(\eta)|^pd\sigma(\eta)\right)^{1/p}.
\end{equation}

In order to have a well-defined problem, the multiplicative operators $f\rightarrow g_kf$ must map the corresponding  space to itself, that is, the functions $g_k$ must be pointwise multipliers of the corresponding spaces. 
We will prove that for the weighted Hardy spaces $H^p(\theta)$ and for the Morrey spaces $HM^{p,s}$, $0<s<n/p$, the space of pointwise multipliers coincide with $H^\infty$. Therefore, in the hypothesis of the corona problems that we will consider, we will  assume that  $g_k\in H^\infty$ for any $k$. 
 
Moreover, we will prove that if $\M_g$ is surjective, then the functions $g_k$ must satisfy condition  \eqref{eqn:coronacond}.

The main object of this paper consist to prove that this condition is also sufficient. To be precise, we will  prove  the following theorem.

\begin{thm} \label{thm:HpMorrey} Let $1< p<\infty$ and $0<s<n/p.$ Let $g_1,\ldots, g_m\in H^\infty$. Then, the following assertions are equivalent:

\begin{enumerate}
\item \label{item:HM1} The functions $g_k$, $k=1,\ldots,m$ satisfy  $\inf\{|g(z)|:\,z\in \B\}>0.$
\item \label{item:HM2} $\M_g$ maps $H^p(\theta)\times \cdots\times H^p(\theta)$ onto $H^p(\theta)$ for any $1<p<\infty$ and any $\theta\in\A_p$.
\item \label{item:HM3} $\M_g$ maps $H^p(\theta)\times \cdots\times H^p(\theta)$ onto $H^p(\theta)$ for some $1<p<\infty$ and some $\theta\in\A_p$.
\item \label{item:HM4} $\M_g$ maps $HM^{p,s}\times\cdots\times  HM^{p,s}$ onto $HM^{p,s}$ for any $1<p<\infty$ and any $0<s<n/p$.
\item \label{item:HM5} $\M_g$ maps $HM^{p,s}\times\cdots\times  HM^{p,s}$ onto $HM^{p,s}$ for some  $1<p<\infty$ and some  $0<s<n/p$.
\end{enumerate}

Moreover, there exists a linear operator $\T_g$ such that $\M_g(\T_g(f))=f$ for all the functions $f$ in one of the above spaces.
\end{thm}

Of course this theorem has interest only for $n>1$, because if $n=1$ the classical corona theorem in $H^\infty$ implies these results. Therefore, we will assume from now on that $n>1$.

We will finish the introduction giving a brief sketch of the proof of Theorem \ref{thm:HpMorrey} and the distribution of the parts of its proof in the different sections of the paper.

As we have already mentioned, it is particularly interesting that the proof of the key point of the theorem, \eqref{item:HM1} implies \eqref{item:HM2}, will be deduced from the proof of the corona problem for $H^2(\theta)$ for any weight $\theta\in \A_2$. In this case $H^2(\theta)$ coincides with a weighted  Besov  space which makes some of the computations easier to deal with. For this reason it is convenient to add to the list 
 of assertions in Theorem \ref{thm:HpMorrey} the following one 
\begin{enumerate}
	\item[(vi)]\label{item:HM6} $\M_g$ maps $H^2(\theta)\times \ldots\times H^2(\theta)$ onto $H^2(\theta)$ for any $\theta\in\A_2$,
\end{enumerate}
the scheme of the proof of the corona theorem for the case of weighted Hardy spaces will be the following:
\begin{center}
\eqref{item:HM2} $\Rightarrow$ \eqref{item:HM3} $\Rightarrow$ \eqref{item:HM1} $\Rightarrow$ (vi) $\Rightarrow $ \eqref{item:HM2}.
\end{center}

Clearly \eqref{item:HM2} $\Rightarrow$ \eqref{item:HM3}. The implication \eqref{item:HM3} $\Rightarrow$ \eqref{item:HM1}, which states that the necessity of the condition \eqref{eqn:coronacond}, is proved in Section \ref{sect:necessary}. The proof of  \eqref{item:HM1} $\Rightarrow$ (vi) is quite technical. In order to make the paper more readable, we prove first this result for the particular case of two generators in Section \ref{sect:corona2}, and next in Section \ref{sect:coronam} we prove the general case. For the case of two generators we will use and minimal solutions of the $\bar\p\p-$equation and Wolff type techniques that allow to estimate the solutions of the corona problem using Carleson  measures for $H^2(\theta)$. 
 
For the proof of the general case, we will consider as usual the Koszul complex with estimates of the involved operators which are suitable for the study of the required continuities. 

Finally, in Section 7 we prove that  (vi) $\Rightarrow $ \eqref{item:HM2}. This result will be a consequence of an extrapolation theorem due to J.L. Rubio de Francia.

The scheme of the proof of the Morrey case is similar and we will show in this case that

\begin{center}
\eqref{item:HM4} $\Rightarrow$ \eqref{item:HM5} $\Rightarrow$ \eqref{item:HM1} $\Leftrightarrow$ \eqref{item:HM2} $\Rightarrow $ \eqref{item:HM4}.
\end{center}

The first implication is obvious, and the proof of the second will be given in Section \ref{sect:necessary}. 
The proof of  \eqref{item:HM2} $\Rightarrow $ \eqref{item:HM4}, given in Section 7, follows from a theorem proved in  \cite{Ar-Mi}.

\section{Preliminaries} \label{sect:prelim}

\subsection{Notations}
In this subsection we  include most of the definitions of operators, spaces of functions and measures that we will use throughout the paper and that have not already been introduced.

As usual, we will adopt the  convention
of using the same letter for various absolute constants whose
values may change in each occurrence, and we will write $A\lesssim
B$ if there exists an absolute constant $M$ such that $A\leq MB$.
 We will say that two quantities $A$ and $B$ are equivalent if both
$A\lesssim B$ and $B\lesssim A$, and, in that case, we will write
$A\approx B$.

\subsubsection{Sets:} For $\z\in \bB$ and $r>0$, let $I_{\z,r}=\{\eta\in \bB: |1-\eta\bar\z |<r\}$. 
When $\z=z/|z|$ and $r=(1-|z|^2)$, we write    $I_z$ instead  of $I_{\z,r}$ .
If $\zeta\in\bB$, and $\alpha>1$, the admissible region is defined by $\Gamma_{\zeta,\alpha}=\{ z\in\B\,;\, |1-z\overline{\zeta}|<\alpha(1-|z|^2)\}$, and if $A\subset \bB$, $T_\alpha(A)=(\cup_{\zeta\notin A} \Gamma_\alpha(\zeta))^c$ is the tent over $A$. When $\alpha=1$, we will write $\Gamma_\zeta=
\Gamma_{\zeta,1}$, and $T(A)=T_1(A)$. If $\zeta\in \bB$, and $r>0$, we will write $\hat I_{\z,r}=T(I_{\zeta,r})$ and if $\z=z/|z|$ and $r=(1-|z|^2)$, we write  $\hat I_z=I_{\z,r}$. 

We denote by $|A|$ the Lebesgue measure of $A$. 
\subsubsection{Differential operators:}
For $1\le j\le n$, let $D_j=\dfrac{\p}{\p z_j}$. 
If $1\le i<j\le n$, let $D_{i,j}$ be the complex-tangential differential operator defined by $D_{i,j}=\bar z_jD_i - \bar z_i D_j$. 
The tangential operators $D_{i,j}$ appear when one computes the coefficients of the form 
\begin{align*}
\p\varphi(z)\land\p|z|^2&=\left(\sum_{i=1}^n D_i\varphi(z) d  z_i \right)\land\left(\sum_{j=1}^n \bar z_jd z_j\right)\\
&=\sum_{1\le i<j\le n} D_{i,j} \varphi(z) d z_i\land dz_j.
\end{align*}

We will consider the usual pointwise norm of the forms 
\begin{align*}|\p \varphi(z)|&=\sum_{j=1}^n|D_j\varphi(z)|,\quad\text{and}\\ |\p_T\varphi(z)|&=|\p\varphi(z)\land\p|z|^2|= \sum_{1\le i<j\le n}| D_{i,j} \varphi(z)|.
\end{align*}

Let $\Ra$ be the radial derivative $\Ra=\sum_{j=1}^n z_j D_j$. For $l> 0$ and $k$ a positive integer, we define 
\begin{equation}\label{eqn:Rlk} 
\Ra_{l}^k=\dfrac{\Gamma(l)}{\Gamma(l+k)} ((l+k-1)\mathcal{I}+\mathcal{R})\ldots 
(l\mathcal{I}+\mathcal{R}),
\end{equation}
where $\mathcal{I}$ denotes the identity operator.

\subsubsection{Integral operators:}
We will denote by $\mathcal{C}$ the Cauchy projection and by $\mathbb{P}$ the Poisson-Sz\"ego projection, that is 
$$\mathcal{C}(\varphi)(z)=\int_{\bB}\frac{\varphi(\z)}{(1-z\bar\z )^n} d\sigma(\z)\,,\qquad \mathbb{P}(\varphi)(z)=\int_{\bB}\varphi(\z)\frac{(1-|z|^2)^n}{|1-z\bar\z |^{2n}} d\sigma(\z).$$ 

In the forthcoming sections, we will use the following kernels and their corresponding integral operators. 
\begin{defn}\label{def:operators}
Let  $N,M, L$ be real numbers satisfying $N>0$ and $L<n$. For $z,w\in \B$, let 
$$\L^{N}_{M,L}(w,z)=\frac{(1-|w|^2)^{N-1}}{|1-z\bar w |^{M}\phi(w,z)^L},$$
where $\phi(w,z)=|1-z\bar w|^2-(1-|w|^2)(1-|z|^2)$.
 
$\L^{N}_{M,L}$ will also denote the corresponding integral operator given by 
$$\psi(z)\rightarrow \L^{N}_{M,L}(\psi)(z)=\int_{\B} \psi(w)\L^{N}_{M,L}(w,z)d\nu(w).$$ 
\end {defn}

The proof of the following result can be found  in Lemma I.1 in \cite{Char}.

\begin {lem} \label{lem:estK} Let  $N,M, L$ be real numbers satisfying $N>0$ and $L<n$. If $n+N-M-2L\neq 0$, then 
$$\int_{\B} \L^{N}_{M,L}(w,z)d\nu(w)\lesssim 1+(1-|z|^2)^{n+N-M-2L},\quad z\in \B.$$
\end {lem}

\begin {defn} We define the type of  the kernel $\L^{N}_{M,L}$ by
$$type(\L^{N}_{M,L})=n+N-M-2L.$$
\end {defn}

\subsection{The Muckenhoupt class $\A_p$ on $\bB$:}

Given a non negative weight $\theta\in L^1(d\sigma)$ and $E$ a measurable set in $\bB$, let $\theta(E)=\int_E\theta d\sigma$. For $z=r\z$, $\z\in \bB$, $0<r<1$, we consider the average function on $\B$ associated to $\theta$ defined by $\Theta(z)=\dfrac{\theta(I_z)}{|I_z|}$, where $I_z=I_{r\z}=\{\eta\in \bB; |1-\eta\bar \z|<1-r^2\}$.

The Muckenhoupt class $\mathcal{A}_p$ on $\bB$, $1< p<\infty$, consists of the non-negative weights $\theta\in L^1(d\sigma)$  satisfying  

\begin{equation}\label{eqn:Ap}
\mathcal{A}_p(\theta)=\sup_{z\in \B}\left(\Theta(z)\right)^{1/p}
\left(\Theta'(z)\right)^{1/p'}<\infty
\end{equation}
where $\theta'=\theta^{-p'/p}$ and $\Theta'(z)=\dfrac{\theta'(I_z)}{|I_z|}$.
Observe that  $\theta\in \mathcal{A}_p$, if and only if, $\theta'\in \mathcal{A}_{p'}$.

If $p=1$, the class  $\mathcal{A}_1$ on $\bB$ is the set of non-negative weights $\theta\in L^1(d\sigma)$  satisfying  
\begin{equation}\label{eqn:A1}
\mathcal{A}_1(\theta)=\left\|\frac{M_{H-L}(\theta)(\z)}{\theta(\z)}\right\|_{L^\infty}<\infty,
\end{equation}
where $M_{H-L}(\theta)(\z)=\sup_r \Theta(r\z)$ is the Hardy-Littlewood maximal function of $\theta$ in $\z$.
\begin{prop} \label{prop:doubling} 
\begin{enumerate}
	\item If $1<p$, then  $\mathcal{A}_1\subset  \mathcal{A}_p$.
	\item If $1<p<\infty$ and $\varphi\in \A_p$, then there exist $1\le q<p$ such that $\varphi\in\A_q$.
	\item For any $\theta\in \mathcal{A}_p$, the measure $\theta d\sigma$ is a doubling measure. In fact, there exist $C>0$ and $0<\lambda< np$ such that for any $\zeta\in \bB$ and any $r>0$, $\theta(I_{\zeta,2r})\leq C 2^\lambda\theta(I_{\zeta,r})$. 
\end{enumerate}
\end{prop}

The proof of (i) can be found in p. 197 \cite{St} and (ii) in p. 202 \cite{St}. The estimate in (iii) with $\lambda\le np$ is proved in p. 196 \cite{St}. This estimate together (ii) gives $\lambda<np$.

\subsection{Holomorphic weighted Hardy spaces}

Let us start recalling some well known facts on the weighted Hardy-Sobolev spaces $H^p(\theta)$, $\theta\in \mathcal{A}_p$.

\begin{prop}\label{prop:Approperties}
Let $\theta$ is in $\A_p$ and $1<p<+\infty$. 
\begin{enumerate}
	\item \label{item:Hp1} If  $M_\alpha(f)(\z)$ is the maximal admissible function $M_\alpha(f)(\z)=\sup\{|f(z)|:z\in \Gamma_{\z,\alpha}\}$, then
$\|f\|_{H^p(\theta)}\approx \|M_\alpha(f)\|_{L^p(\theta)}.$
	\item There exist $1<p_1<p< p_2$ such that $H^{p_2}\subset H^p(\theta)\subset H^{p_1}$.
	\item If $1<p<\infty$ and $\theta\in \mathcal{A}_p$, then the Cauchy projection $\mathcal{C}$  maps $L^p(\theta)$ to $H^p(\theta)$.
	\item The dual of $H^p(\theta)$ can be identified with $H^{p'}(\theta')$ with the pairing given by   
\begin{equation}\label{eqn:pairingS} \langle f, g\rangle_S=\lim_{r\rightarrow 1}\int_S f_r \bar g_r d\sigma,
\end{equation} 
where $f_r(\z)=f(r\z)$.
\end{enumerate}
\end{prop} 

The proof of part \eqref{item:Hp1} can be found in Section 5 in \cite{Lu}. 

The following result (see \cite{Ca-Or2}) gives that the weighted space $H^2(\theta)$ can be considered as a weighted Besov space.

\begin{prop} \label{prop:charhp1} Let $\theta\in \A_2$ and let $k$ be any positive integer. The following assertions are equivalent
\begin{enumerate}
	\item $f\in H^2(\theta)$.
	\item $\displaystyle{(1-|z|^2)^{k-1/2}\Theta(z)^{1/2}|(\I+\Ra)^k f(z)|\in L^2(\B)}$.
\end{enumerate}
\end{prop}

Of course, in the last expression we can replace the oprator $(\I+\Ra)^k$ by an operator $\Ra_l^k$ defined in \eqref{eqn:Rlk}.

 Observe that Lemma 3.6 in \cite{Ah-Br}, gives that if $\alpha<\beta$, and $h$ is a holomorphic function on $\B$, there exists $C>0$ such that for any $\zeta\in\bB$,
$$\int_{\Gamma_{\zeta,\alpha}} |\p_T h(z)|^2 (1-|z|^2)^{-n} d\nu(z)\lesssim  \int_{\Gamma_{\zeta,\beta}} |\Ra h(z)|^2 (1-|z|^2)^{1-n} d\nu(z).$$
Consequently, multiplying by $\theta(\zeta)$ and integrating we obtain
\begin{equation}\label{eqn:admissible}
\int_{\B}|D_{i,\,j}h(z)|^2\Theta(z)d\nu(z)\lesssim  \int_{\B}|\Ra  h(z)|^2(1-|z|^2)\Theta(z)d\nu(z)
\end{equation}
In particular, we have from this observation and Proposition \ref{prop:charhp1} that

\begin{prop} \label{prop:charhp} If $f\in H^2(\theta)$, then 
$$\left\|\Theta(z)^{1/2}\left(|\p_T f(z)|+ (1-|z|^2)^{1/2}|\p f(z)|\right)\right\|_{L^2(\B)}\lesssim \|f\|_{ H^2(\theta)}.$$
\end{prop}

\subsection{Holomorphic Morrey spaces $0<s<n/p$.}

The following embedding is a consequence of H\"older's iequality, and will be used in the forthcoming sections.

\begin{prop} \label{prop:embed:MHp} If $1<p<\infty$ and $0< s\le n/p$, then $H^{n/s}\subset HM^{p,s}\subset H^p$.
\end{prop}

\section{Pointwise multipliers and Carleson measures} \label{sect:multcar}

In this section we will show that the space of pointwise multipliers of the holomorphic weighted Hardy spaces $H^p(\theta)$ and of the holomorphic Morrey spaces $M^{p,s}$ coincide with $H^\infty$. We will also give examples of Carleson measures for $H^p(\theta)$, which will play an important role in the proofs of the main theorems.

\subsection{Pointwise multipliers and Carleson measures on weighted Hardy spaces}

\begin{prop} A holomorphic function $g$ on $\B$ is a pointwise multiplier of $H^p(\theta)$ if and only if $g\in H^\infty$.
\end{prop}

\begin{proof} It is clear that, if $g\in H^\infty$, then $g$  is a pointwise multiplier of $H^p(\theta)$. 

The converse assertion is a consequence of   the inequality
$\|g^m\|_{H^p(\theta)}\le \|\M_g\|^m \|1\|_{H^p(\theta)}$, where $\|\M_g\|$ denotes the norm of the operator $\M_g(f)=gf$, and that 
$\|g\|_{H^\infty}\lesssim \sup_{m} \|g^m\|_{H^p(\theta)}^{1/m}$.
\end{proof}

\begin{prop}\label{prop:test}
Let $1<p<+\infty$, $\theta\in \A_p$, $\lambda$ the constant in the doubling condition of $\theta$, $\theta(I_{\z,2r})\lesssim C2^\lambda \theta(I_{\z,r})$, and $Np> \lambda$.   Let $f_z(w)=\dfrac1{(1-w\overline{z})^N}$, for $z\in \B$. We then have
$$\|f_z\|_{H^p(\theta)}^p \lesssim  \frac{\theta(I_z)}{(1-|z|^2)^{Np}}.$$
\end{prop}
\begin{proof} By Proposition \ref{prop:doubling}, the measure $\theta d\sigma$ is a doubling measure and $\theta(2I_z)\lesssim 2^{\lambda} \theta(I_z)$ with $\lambda<np$. 
If $z=|z|\z$, and $2 I_z=I_{\z,2(1-|z|^2)}$, we have 
\begin{align*}
 \left\|\frac1{(1-w\overline{z})^N}\right\|_{H^p(\theta)}^p &= \int_{\bB}\frac 1{|1-\zeta\overline{z}|^{Np}}\theta(\zeta)d\sigma(\zeta)\\
\quad&\lesssim  \sum_{k\geq0} \frac{\theta(2^k I_z)}{2^{kNp}(1-|z|^2)^{Np}}\lesssim  \sum_{k\geq0} \frac{2^{k\lambda}\theta(I_z)}{2^{kNp}(1-|z|^2)^{Np}}\\\quad&\lesssim  
\frac{ \theta(I_z)}{ (1-|z|^2)^{Np}},
\end{align*}
where in last estimate we have used the fact that $Np>\lambda$.
\end{proof}
We recall that a positive Borel measure $\mu$ on $\B$ is a Carleson measure for a space $X^p$ of functions in $\B$, if there exists $C>0$ such that for any $f\in X^p$, 
\begin{equation}\label{eqn:CMHp}
\int_{\B} |f(z)|^pd\mu(z)\leq C\|f\|_{X^p}^p.
\end{equation}
As in the unweighted case, when $1<p<+\infty$, $\theta$ is an $\A_p$ and $X^p$ is either $H^p(\theta)$ or the space $\mathbb{P}[L^p(\theta)]$, these measures can be characterized in terms of conditions on tent of balls.

The space $\mathbb{P}[L^p(\theta)]$ is normed by $\|u\|_{\mathbb{P}[L^p(\theta)]}=\|f\|_{L^p(\theta)}$, where $u=P[f]$, and we recall that $\|u\|_{\mathbb{P}[L^p(\theta)]}\approx \| M_\alpha[u]\|_{L^p(\theta)}$.
 
 We then have:

\begin{prop} \label{prop:CarHp1} Let $1<p<+\infty$, $\mu$ a positive Borel measure on $\B$ and $\theta$ be a weight in $\A_p$. Then the following assertions are equivalent:
\begin{itemize}
\item[(i)] $\mu$ is a Carleson measure for $\mathbb{P}[L^p(\theta)]$.
\item[(ii)] $\mu$ is a Carleson measure for $H^p(\theta)$.
 \item[(iii)]There is a constant $C$ such that for all $z\in \B$, $\mu(\hat{I}_z)\le C \theta(I_z).$ 
\end{itemize}
\end{prop}

\begin{proof}
From Theorem 5.6.8 in \cite{Ru}, and Proposition \ref{prop:Approperties}, it is immediate to deduce that for any function $f\in H^p(\theta)$, $f=\mathcal{C }[f^*]=\mathbb{P}[f^*]$, and $\|f\|_{H^p(\theta)}\approx  \|f\|_{\mathbb{P}[L^p(\theta)]}\approx  \|f^*\|_{L^p(\theta)}$, where $f^*$ are the boundary values of the function $f$. Hence, any Carleson measure for $\mathbb{P}[L^p[\theta]]$ is also a Carleson measure for $H^p(\theta)$, and, in consequence, (i) implies (ii).

Next, assume that (ii) holds, and for $z=|z|\z$, and $N>0$ big enough, let $f_z(w)=\frac1{(1-w\overline{z})^N}$. We then have that for any $w\in \hat I_z$, $|1-w\overline{z}|\lesssim  1-|z|^2$. Thus Proposition \ref{prop:test} gives
\begin{align*}\frac{\mu(\hat I_z)}{(1-|z|^2)^{Np}} \lesssim  \int_{\B} \frac{d\mu(w)}{|1-w\overline{z}|^{Np}}\lesssim  \|f_z\|_{H^p(\theta)}^p\lesssim  \frac{ \theta(I_z)}{ (1-|z|^2)^{Np}} .\end{align*}

So we are left to show that (iii) implies (i).
We have 
$$\int_{\B}|\mathbb{P}[f]|^pd\mu= p\int_0^{+\infty} \mu(\{ \mathbb{P}[f]>\lambda\})\lambda^{p-1} d\lambda.$$
But $\{ \mathbb{P}[f]>\lambda\}\subset T(\{ \zeta\,;\, M_\alpha \mathbb{P}[f]>\lambda\})$. Since $A_\lambda=\{ \zeta\,:\, M_\alpha\mathbb{P}[f]>\lambda\}$ is an open set, $A_\lambda=\cup I_{\zeta, r_\zeta}$, and for any compact $K\subset A_\lambda$, there exists a finite subfamily of pairwise disjoint open balls $I_{\zeta_i,r_{\zeta_i}}$ such that $K\subset\cup_1^M I_{\zeta_i,3r_{\zeta_i}}$.
Consequently,
\begin{align*}\begin{split}
\mu(T(K))&\leq \sum_{i=1}^M \mu(T(I_{\zeta_i,3r_{\zeta_i}}))\lesssim  \sum_{i=1}^M \theta(I_{\zeta_i,3r_{\zeta_i}})\\
& \lesssim \sum_{i=1}^M \theta(I_{\zeta_i,r_{\zeta_i}})=\theta(\cup_{i=1}^M I_{\zeta_i,r_{\zeta_i}})\lesssim  \theta(A_\lambda),\end{split}\end{align*}
and
$$\int_{\B}|\mathbb{P}[f]|^pd\mu\lesssim  \int_0^{+\infty} \theta(A_\lambda) \lambda^{p-1} d\lambda\approx  \int_{\bB } |M_\alpha \mathbb{P}[f]|^p d\theta\lesssim  \int_{\bB } |f|^p d\theta
.$$
\end{proof}

\begin{prop} \label{prop:CarHp2} If $g\in H^\infty$, then  $|\p  g(z)|^2(1-|z|^2)\Theta(z)d\nu(z)$ and $|\p_T g(z)|^2\Theta(z)d\nu(z)$ are  Carleson measures for $H^2(\theta)$.
\end{prop}
\begin{proof} Let  $f\in H^2(\theta)$. 
Since $gf\in H^2(\theta)$, by Proposition \ref{prop:charhp} we have that 
$$\int_{\B}|\p(gf)(z)|^2(1-|z|^2)\Theta(z)d\nu(z)\approx  \|gf\|_{H^2(\theta)}^2\lesssim  \|f\|_{H^2(\theta)}^2.$$
This observation, the fact that $f\in H^2(\theta)$, $g\in H^\infty$ and
$f\p g=\p(gf)-g\p f$, give that,
\begin{align*}
&\int_{\B}|f(z)|^2|\p g(z)|^2(1-|z|^2)\Theta(z)d\nu(z)\\
&\quad\leq\int_{\B}|\p(gf)(z)|^2(1-|z|^2)\Theta(z)d\nu(z)\\
&\qquad +\int_{\B}|\p f(z)|^2| g(z)|^2(1-|z|^2)\Theta(z)d\nu(z)\\
&\quad\lesssim  \|f\|_{H^2(\theta)}^2.
\end{align*}
We now deal with the second assertion. We have that $fD_{i,\,j}g=D_{i,\,j}(gf)-gD_{i,\,j}f$.
Hence, 
\begin{align*}\begin{split}
\int_{\B}|f(z)|^2|\p_T g(z)|^2&\Theta(z)d\nu(z)
\leq\int_{\B}|\p_T(gf)(z)|^2\Theta(z)d\nu(z)\\
&+\int_{\B}|\p_Tf(z)|^2| g(z)|^2(1-|z|^2)\Theta(z)d\nu(z).
\end{split}\end{align*}

Applying Proposition \ref{prop:embed:MHp} to both $f$ and $gf$ in the preceding estimate with $g\in H^\infty$, we obtain that
\begin{align*}\begin{split}
&\int_{\B}|f(z)|^2|\p_T g(z)|^2\Theta(z)d\nu(z)\lesssim 
\|gf\|_{H^2(\theta)}^2+\|f\|_{H^2(\theta)}^2\lesssim  \|f\|_{H^2(\theta)}^2.
\end{split}\end{align*}
\end{proof}

As a consequence of  Propositions \ref{prop:CarHp1} and \ref{prop:CarHp2} we have:

\begin{prop} \label{lem:estKw12} Let $\theta\in \A_2(\bB)$ and $g\in H^\infty$. Let 
$$d\mu_{g,\theta}(z)=\Theta(z)\left( (1-|z|^2)|\p g(z)|^2+|\p_T g(z)|^2\right)d\nu(z).$$

If  $\varphi\in L^2(\theta)$ , then 
$$\int_{\B} \left|\mathbb{P}(\varphi)(z)\right|^2 d\mu_{g,\theta}(z)\lesssim \|\varphi\|_{L^2(\theta)}^2.$$
\end{prop}

\subsection{Multipliers on Morrey spaces}
The next result gives a characterization of the pointwise multipliers of $ HM^{p,s},$ for $0<s\le n/p.$

\begin{prop}\label{lem:lem11} If $1<p<\infty$ and $0<s\le n/p,$ a function $g$ is a pointwise multiplier of $HM^{p,s},$ if and only if $g\in H^\infty.$
\end{prop}

\begin{proof} It is clear that if $g\in H^\infty,$ then $\|gf\|_{M^{p,s}}\le \|g\|_{\infty}\|f\|_{M^{p,s}}.$

To prove the converse result, note that for a positive integer $m,$
$$\|g^m\|_{p}\le \|g^m\|_{M^{p,s}}\le \|\M_g\|^m\|1\|_{M^{p,s}},$$
where $\|\M_g\|$ denotes the norm of the operator $f\rightarrow gf.$
Therefore, $\|g\|_{pm}\lesssim \|\M_g\|.$ Since $\|g\|_{\infty}=\lim_{m\rightarrow\infty}\|g\|_{pm}$ we obtain the result.
\end{proof}

\section{Necessary conditions on the  corona problem}\label{sect:necessary}

\subsection{Necessary conditions for weighted Hardy spaces}
In order to obtain necessary conditions on the corona problem we recall the following lemma which has been proved in \cite{Ca-Or2}
\begin{lem}\label{lem:pointwise} 
Let $1<p<+\infty$ and $\theta$ a weight in $\A_p$. There exists $C>0$ such that for any holomorphic function $f$ in $\B$, and any $z=|z|\z$,
$$|f(z)|\leq C\left( |f(0)|+\int_{1-|z|^2}^1 \frac{dt}{\theta(I_{\z,t})^{1/p}\, t}
\|f\|_{H^p(\theta)}\right).$$
\end{lem}
As a corollary we obtain
\begin{cor}\label{cor:pointwise2}
Let $1<p<+\infty$ and $\theta$ a weight in $\A_p$. There exists $C>0$ such that for any holomorphic function $f$ in $\B$, and any $z\in \B$,
$$|f(z)|\leq C \frac{\left(\theta^{-p'/p}(I_z)\right)^{1/p'}}{(1-|z|^2)^n}\|f\|_{H^p(\theta)}.$$
\end{cor}
\begin{proof} Let $z=|z|\z\ne 0$.
The fact that $\theta$ is in $\A_p$ gives that $\theta^{-p'/p}$ is in $\A_{p'}$, it satisfies a doubling condition of order $\lambda<np'$, and consequently
\begin{align*}
\int_{1-|z|^2}^1 \frac{1}{\theta(I_{\z,t})^{1/p} } \frac{dt}{t}&\approx  \int_{1-|z|^2}^1 \left( \frac{\theta^{-p'/p}(I_{\z,t}) }{t^n}\right)^{1/p'}\frac1{t^{n/p}}\frac{dt}{t}\\
&\lesssim  \sum_{k\geq 0} 2^{k(\lambda/p'-n)} \frac{\left(\theta^{-p'/p}(I_z)\right)^{1/p'} }{(1-|z|^2)^n}\\
&\approx \frac{\left(\theta^{-p'/p}(I_z)\right)^{1/p'} }{(1-|z|^2)^n}.
\end{align*}

In order to finish we just have to show that  $|f(0)|\lesssim  \dfrac{\left(\theta^{-p'/p}(I_z)\right)^{1/p'}}{(1-|z|^2)^n}$. This is a consequence from the fact that
$$|f(0)|\leq \int_{\bB}|f|d\sigma=\int_{\bB}|f|\theta^{1/p}\theta^{- 1/p}d\sigma\lesssim  \|f\|_{H^p(\theta)},$$
and 
$$ 1\lesssim \frac{\left(\theta(I_z)\right)^{1/p}\left(\theta^{-p'/p}(I_z)\right)^{1/p'}}{(1-|z|^2)^n}\lesssim \frac{\left(\theta^{-p'/p}(I_z)\right)^{1/p'}}{(1-|z|^2)^n}.$$
\end{proof}

\begin{prop} \label{prop:necH} Let $g_1,\ldots, g_m\in H^\infty$ and $\theta\in\A_p$. If the operator $\M_g 
:H^p(\theta)\times\cdots\times H^p(\theta)\rightarrow H^p(\theta)$ is onto, then $g=(g_1,\ldots, g_m)$ satisfies  $\inf_{z\in\B}|g(z)|>0$.
\end{prop}
\begin{proof}
 We first observe that by the open map Theorem, for every $f\in H^p(\theta)$, there exists functions $f_i \in H^p(\theta)$, $i=1,\ldots, m$, such that
 \begin{itemize}
 \item[(i)] ${\displaystyle{f=\sum_{i=1}^mf_ig_i}}$.
 \item[(ii)] $\|f_i\|_{H^p(\theta)}\lesssim  \|f\|_{H^p(\theta)}$, $i=1,\ldots,m$.
 \end{itemize}
 If $z\in \B$, let $f_z(w)=\dfrac1{(1-w\overline{z})^N}$, where $N>0$ is to be chosen. We then have that there exist $f_i$, $i=1,\ldots,m$, satisfying conditions (i) and (ii) above.
 Therefore, Corollary \ref{cor:pointwise2} and Proposition \ref{prop:test}, if $N$ is big enough, give
 \begin{align*}
\frac1{(1-|z|^2)^N}=|f_z(z)|&\leq \sum_{i=1}^m |f_i(z)||g_i(z)| \\
&\lesssim  \frac{\left(\theta^{-p'/p}(I_z)\right)^{1/p'}}{(1-|z|^2)^n}
\|f_z\|_{H^p(\theta)}\sum_{i=1}^m |g_i(z)|\\&\lesssim 
\frac{\left(\theta^{-p'/p}(I_z)\right)^{1/p'}}{(1-|z|^2)^n}\frac{\theta(I_z)^{1/p}}{(1-|z|^2)^N} \sum_{i=1}^m |g_i(z)|,
 \end{align*}
 and since $\theta$ is in $\A_p$, we obtain that $\displaystyle{1\lesssim  \sum_{i=1}^m |g_i(z)|}$.
\end{proof}

\subsection{Necessary conditions for Morrey spaces}

The next lemma gives a pointwise estimate for $f\in HM^{p,s}.$

\begin{prop}\label{prop:morreygr} Let $1\le p<\infty,$ $0<s<n/p,$ $f$ in $ HM^{p,s},$ and $z\in B.$ Then
$|f(z)|\lesssim \|f\|_{p,s}(1-|z|^2)^{-s}.$
\end{prop}

\begin{proof} By the Cauchy formula
$$|f(z)|\le \int_S\frac{|f(\eta)|}{|1-z\bar\eta |^n}d\sigma(\eta).$$

Assume $z\ne 0,$ and let $\z=z/|z|.$ For a positive integer $j$, let $I_j=\{\eta\in \bB: |1-\eta\bar\z|\le 2^{j+1}(1-|z|^2)\}.$

Therefore,
\begin{align*}
|f(z)|&\le \int_{I_1}\frac{|f(\eta)|}{|1-z\bar\eta |^n} d\sigma(\eta) +\sum_{j>1}\int_{I_{j+1}-I_j}\frac{|f(\eta)|}{|1-z\bar\eta |^n}d\sigma(\eta)\\
&\lesssim\sum_{j\ge 1}(2^j(1-|z|^2))^{-n}\int_{I_{j+1}}|f(\eta)|d\sigma(\eta).
\end{align*}

By H\"older's inequality 
$$|f(z)|\lesssim\sum_{j\ge 1}(2^j(1-|z|^2))^{-s}\|f\|_{p,s}\lesssim
(1-|z|^2)^{-s}\|f\|_{p,s}$$
which concludes the proof. 
\end{proof}

\begin{prop} \label{prop:necM} Assume that $g_1,\ldots,g_m\in H^\infty$. If for some $1<p<\infty$ and $0<s<n/p$ the map $\M_g:HM^{p,s}\times\cdots\times HM^{p,s}\rightarrow HM^{p,s}$ is  surjective, then  $\inf\{|g(z)|: z\in \B\}>0.$
\end{prop}

\begin{proof} 
By the open map Theorem, for every function $\,f\,$ in $HM^{p,s}.$
there exist functions $\,f_1,\ldots,f_m\,$ in $HM^{p,s}$, such that
$$\sum_{k=1}^m\,g_k f_k\,=f\qquad \text{and}\qquad 
\|f_k\|_{M^{p,s}}\,\lesssim\,\|f\|_{M^{p,s}}.
$$

By Proposition \ref{prop:morreygr}
\begin{equation}\label{eqn:nec1}
|f(z)|\,\le\,\sum_{i=1}^m\,|f_i(z)|\,|g_i(z)|\,\lesssim\, 
\|f\|_{M^{p,s}}(1-|z|^2)^{-s}\,\sum_{i=1}^m\,|g_i(z)|.
\end{equation}
For $N>s$, consider the function $\,f_z(w)\,=\,(1-w\bar z )^{-N}$. Since, by Proposition \ref{prop:embed:MHp}, $H^{n/s}\subset HM^{p,s}$, we have 
$$\|f_z\|_{M^{p,s}}\lesssim \|f_z\|_{H^{n/s}}\approx (1-|z|^2)^{s-N}.$$

Therefore, by \eqref{eqn:nec1} 
\begin{align*}
(1-|z|^2)^{-N}=f_z(z)&\lesssim \|f_z\|_{M^{p,s}}(1-|z|^2)^{-s}\,\sum_{i=1}^m\,|g_i(z)|\\
&\lesssim (1-|z|^2)^{-N}\,\sum_{i=1}^m\,|g_i(z)|,
\end{align*}
which proves the result.
\end{proof}

\section{The $H^2(\theta)$-corona theorem for 2 generators.} \label{sect:corona2}

Throughout this section we will assume that the functions $g_1,g_2\in H^\infty$ satisfy 
$\inf_{z\in \B}|g(z)|>0$. 

We want to prove that the operator  $\M_g$ defined by $\M_g(f_1,f_2)=g_1f_1+g_2f_2$
 maps $H^2(\theta)\times H^2(\theta)$ onto $H^2(\theta)$ for all the weights $\theta \in \A_2$.

Let $g=(g_1,g_2)$ and let $G=(G_1,G_2)$ where $G_j=\dfrac{\bar g_j}{|g|^2},\,j=1,2$.
An easy computation proves that
\begin{equation}\label{eqn:omega1}
\bar\p G_1=-g_2\Omega,\quad \bar \p G_2=g_1\Omega,
\end{equation}
where
\begin{equation}\label{eqn:omega}
\Omega=\frac{\overline{\,\, g_1 \p g_2- g_2 \p  g_1\,\,}}{|g|^4}=G_1\bar\p G_2-G_2\bar\p G_1.
\end{equation}
Clearly $g_1G_1+g_2G_2=1$, $\bar\p \Omega=0$ and $\|G f\|_{L^2(\theta)}\lesssim \|f\|_{H^2(\theta)}$. 

Since the functions $G_jf$ are not holomorphic on $\B$, we must correct them by using a solution of a $\bar\p$ problem. Since $\bar\p (\Omega f)=0$ for any $f\in H^2(\theta)$, we will choose a suitable  integral operator $\K$ such that $\bar\p \K(\Omega f)=\Omega f $  and such that the linear operator 
\begin{equation}\label{eqn:coronaT2}
\T_g(f)=G f+ g^{\perp}\K(\Omega f), \quad g^\perp=(g_2,-g_1)
\end{equation}
maps   $H^2(\theta)$ to $H^2(\theta)\times H^2(\theta)$.

It is clear by construction that the components of $\T_g(f)$ are holomorphic functions on $\B$ and that $\M_g(\T_g(f))=f$.

In order to choose a suitable operator $\K$, let 
\begin{equation*}
\K^N_0(w,z)=\,\sum_{k=0}^{n-1}\,c_{k,N}\frac{(1-|w|^2)^{N+k}( s\land (\bar\p_w\, s)^{n-1-k})(w,z)}{(1-z\bar w )^{n+N}(1-w \bar z)^{n-k}}\, \land (\gamma(w))^k
\end{equation*}
where $\bar\p=\bar\p_w$ (differential respect $w$), $\gamma(w)=\bar\p\,\dfrac{\p|w|^2}{1-|w|^2}$ and 
$s(w,z)=(1- w \bar z)\p|w|^2 -(1-|w|^2)\p_w(w\bar z)$.

It is well-known that the corresponding integral operators associated to these kernels, also denoted by $\K_0^N$, that is, if $\vartheta$ is  $(0,1)$-form,
$$\K_0^N(\vartheta)(z)=\int_{\B} \K_0^N(w,z) \vartheta(w)d\nu(w),$$  solve the $\bar\p$-equation or the $\bar\p_b$-equation in the unit ball of $\C^n$ (see for instance   \cite{Sko} or \cite{Char})

The following proposition gives the main properties of these operators. In particular  it gives a decomposition of  $\K_0^N(\vartheta)$ as a sum of two functions. The first one is an antiholomorphic function on $\B$, and the other term involves $\p\vartheta$. The main advantatge of this last term is that if $\vartheta$ is the form $\Omega$ defined in \ref{eqn:omega}, then, by Proposition \ref{prop:CarHp2}, we obtain expressions like  $\Theta(z)|\p\Omega(z)|^2(1-|z|^2)d\nu(z)$ or  $\Theta(z)|\p\Omega(z)\land\p|z|^2\land\bar\p|z|^2|^2d\nu(z)$ that are Carleson measures for $H^2(\theta)$, and that will play an important rol in the calculus of the estimates.
 
\begin{prop} \label{prop:KN1} Let $\vartheta$ be a $(0,1)$-form with coefficients in $\mathcal{C}^1(\bar\B)$. Then, for each positive integer $N$, there exist integral operators $\Q_{0}^{N,1}$ and $\Q_{0}^{N,2}$ satisfying the following properties.
\begin{enumerate}
	\item \label{item:KN12} $\K_0^N(\vartheta)=\Q_{0}^{N,1}(\vartheta)+\Q_{0}^{N,2}(\p\vartheta)$.
	\item \label{item:KN11}$\bar\p\K^N_0(\vartheta)=\vartheta$ if $\bar\p \vartheta=0$.
	\item \label{item:KN13} The function $\Q_{0}^{N,1}(\vartheta)(z)$ is antiholomorphic on $\B$ and for $\z\in \bB$ $\displaystyle{\Q_{0}^{N,1}(\vartheta)(\z)=c\sum_{k=1}^{n+N}\int_{\B}\frac{(1-|w|^2)^N\vartheta(w) \land\p(w\bar \z)\land(\p\bar\p|w|^2)^{n-1}}{(1-w \bar \z)^k}}$.
	\item \label{item:KN14} For $\z\in\bB$,   
\begin{align*}|\Q_{0}^{N,2}(\p\vartheta)(\z)|&\lesssim \int_{\B} \frac{(1-|w|^2)^{N+1}|\p\vartheta(w)|}{|1-w\bar \z|^{n+N}}d\nu(w)\\ &\quad +\int_{\B}\frac{(1-|w|^2)^{N}|\p\vartheta(w)\land\p|w|^2\land\bar\p|w|^2|}{|1-w\bar \z|^{n+N}}d\nu(w).
\end{align*}
\end{enumerate}
\end{prop}

The proof of assertions \eqref{item:KN12}, \eqref{item:KN13} and \eqref{item:KN14} can be found in \cite{An2}. Assertion \eqref{item:KN12} is proved in  \cite{Sko} and \cite{Char}.

We want to prove  that for $N>0$ big enough, $\T_g^N$ maps $H^2(\theta)$ to $H^2(\theta)\times H^2(\theta)$,  that is $\|\T_g^N(f)\|_{L^2(\theta)}\lesssim \|f\|_{H^2(\theta)}$  with a constant depending of $n,N$, $g$ and $\theta$.

Since $|\T_g^N(f)|\le |G||f|+|g^\perp||\K^N(\Omega f)|\le \|G\|_\infty|f|+\|g\|_\infty|\K^N(\Omega f)|$,  
 we only need to  prove that for  $N>0$ large enough we have the estimate $\|\K^N(\Omega f)\|_{ L^2(\theta)}\lesssim \|f\|_{H^2(\theta)}$.

 Since $\K^N_0(\Omega f)=\Q_{0}^{N,1}(\Omega f) + \Q_{0}^{N,2}(\p(\Omega f))$, we will need the following estimates of $\Omega f$ and of $\p(\Omega f)$.

\begin{lem} \label{lem:estomeg1} Let $\Omega$ as in \eqref{eqn:omega}. Then  
\begin{align*}
|(\Omega f)(w)|&\lesssim |\p g(w)||f(w)|\\
|\p(\Omega f)(w)|&\lesssim |\p g(w)|^2|f(w)|+|\p g(w)||\p f(w)|.\\
|\p(\Omega f)(w)\land\p|w|^2\land\bar\p|w|^2|&\lesssim |\p_T g(w)|^2|f(w)|+|\p_T g(w)||\p_T f(w)|.
\end{align*}
\end{lem}

\begin{proof} All the above estimates follows from the definition of $\Omega$ and the formulas
$|\overline{\p  g_k(w)}\land \bar\p|w|^2|= |\p_T g_k(w)|$ and  $|\p f(w)\land \p|w|^2|= |\p_T f(w)|$
\end{proof}

We will obtain the estimates of the norm of  $\K^N_0(\Omega f)$ in $L^2(\theta)$ by duality. We will need the following lemmas. 

\begin{lem} \label{lem:estKwf0}
 If $N>0$, then 
$$
|\K^N_0(\Omega f)(\z)|
\lesssim |\Q_0^N(\Omega f)(\z)+\L^{N+1}_{n+N,0}(W)(\z),$$
where
\begin{align*}
W(w)&=(1-|w|^2)\left(|\p g(w)|^2|f(w)|+|\p g(w)||\p f(w)|\right)\\
&\quad+|\p_T g(w)|^2|f(w)|+|\p_T g(w)||\p_T f(w)|.\end{align*}
\end{lem}

\begin{proof} The proof is a consequence of  Proposition \ref{prop:KN1} and Lemma \ref{lem:estomeg1}.
\end{proof}

\begin{lem} \label{lem:estKwf} Let $\psi$ be a continuous function on $\bB$.
 If $N>n$, then 
\begin{align*}
&\left|\int_{\bB} \K^N_0(\Omega f)\,\psi d\sigma\right|\\
&\lesssim 
\sum_{k=1}^{n+N}\int_{\B} |\p g(w)||f(w)|(1-|w|^2)^N\left|\int_{\bB}\frac{\psi(\z)}{(1-w \bar \z)^k}d\sigma(\z)\right|d\nu(w)\\
 &+\int_{\B}(1-|w|^2)\left(|\p g(w)|^2|f(w)|+|\p g(w)||\p f(w)|\right)\mathbb{P}(|\psi|)(w)d\nu(w)\\
 &+\int_{\B}\left(|\p_T g(w)|^2|f(w)|+|\p_T g(w)||\p_T f(w)|\right)\mathbb{P}(|\psi|)(w)d\nu(w).
\end{align*}
\end{lem}

\begin{proof} The proof is a consequence of  Lemma \ref{lem:estKwf0}, Fubini's Theorem and the estimate $1-|w|^2\le 2|1-\z\bar w|$.
\end{proof}

\begin{lem} \label{lem:estKwl1} If $\theta\in \A_2(\bB)$, then for any positive integer $N$ we have
$$\sum_{k=1}^{n+N}\int_{\B} \Theta(w) (1-|w|^2)^{2N-1}\left|\int_{\bB}\frac{\psi(\z)}{(1-w \bar \z)^k}d\sigma(\z)\right|^2d\nu(w)\lesssim \|\psi\|_{L^2(\theta)}^2.$$
\end{lem}

\begin{proof} Observe that $\displaystyle{\|\mathcal{C}(\psi)\|_{ H^2(\theta)}=\left\|\int_{\bB} \dfrac{\psi(\z)}{(1-w\bar\z)^{n}}d\sigma(\z)\right\|_{ H^2(\theta)} \lesssim  \|\psi\|_{ L^2(\theta)}}$. If $n\le k\leq n+N$, then
$$\int_{\bB}\frac{\psi(\z)}{(1-w \bar \z)^k}d\sigma(\z) =
\Ra_n^{k-n}\mathcal{C}(\psi)(z).
$$ 
(see \eqref{eqn:Rlk} for the definition of $\Ra_n^{k-n}$).
Therefore, the result is a consequence of Proposition \ref{prop:charhp1}.

In order to prove the case $1\le k<n$, observe that 
$$\mathcal{C}(\psi)(z) =
\Ra_k^{n-k}\int_{\bB}\frac{\psi(\z)}{(1-w \bar \z)^k}d\sigma(\z).$$
Therefore,
$$\left\|\int_{\bB}\frac{\psi(\z)}{(1-w \bar \z)^k}d\sigma(\z)\right\|_{H^2(\theta)}\lesssim \|\mathcal{C}(\psi)\|_{H^2(\theta)}
$$
and by Proposition \ref{prop:charhp1} we conclude the proof.
\end{proof}

\begin{prop} \label{prop:corona2} If $N>n$ and $\theta\in\A_2$, then $\|\K^N_0(\Omega f)\|_{ L^2(\theta)}\lesssim \|f\|_{H^2(\theta)}$.
\end{prop}

\begin{proof} Let $\theta'=\theta^{-1}\in \mathcal{A}_2$ and let $\Theta$ and $\Theta'$ be the corresponding  averages of $\theta$ and $\theta'$. 
Let also 
$$d\mu_{g,\theta}(z)=\Theta(z)\left( (1-|z|^2)|\p g(z)|^2+|\p_T g(z)|^2\right)d\nu(z).$$
We recall that by Proposition \ref{lem:estKw12},  $\mu_{g,\theta}$ is a Carleson measure for $H^2(\theta)$.

Let
$$\Psi(w)=\sum_{k=1}^{n+N}  (1-|w|^2)^{N-1/2}\left|\int_{\bB}\frac{\psi(\z)\theta(\z)}{(1-w \bar \z)^k}d\sigma(\z)\right|.$$

By Lemma \ref{lem:estKwf}, H\"older's Inequality and the fact that $\Theta(z)^\frac12\Theta'(z)^\frac12\approx 1$, we have 
\begin{align*}
&\int_{\bB} \K^{N}_0(\Omega f)\psi \,d\sigma
=\left|\int_{\bB} \left(\Q_{0}^{N,1}(\Omega f)+\Q_{0}^{N,2}(\p(\Omega f))\right)\,\psi  d\sigma\right|\\
&\lesssim \left(\int_{\B} |f|^2d\mu_{g,\theta}\right)^{1/2}\left(\int_{\B}\Theta'|\Psi|^2 d\nu\right)^{1/2}\\
&\,\,+\left(\int_{\B}|f|^2d\mu_{g,\theta}\right)^{1/2}\left(\int_{\B}|\mathbb{P}(\psi)|^2d\mu_{g,\theta'}\right)^{1/2}\\
&\,\,+\left(\int_{\B}((1-|w|^2)|\p f|^2+|\p_T f|^2)\Theta d\nu\right)^{1/2}\left(\int_{\B}|\mathbb{P}(\psi)|^2d\mu_{g,\theta'}\right)^{1/2}.
\end{align*}

Therefore, Propositions \ref{prop:charhp} and \ref{lem:estKw12}, and Lemma \ref{lem:estKwl1} give that 
$$\left|\int_{\bB} \K_{0}^{N}(\p(\Omega f))\,\psi  d\sigma\right|\lesssim \|f\|_{H^2(\theta)}\|\psi\|_{L^2(\theta')}$$
which concludes the proof.
\end{proof}

As a consequence of the above proposition we have:
\begin{thm} \label{corh22} Let $g_1,g_2\in H^\infty$ satisfying $\inf\{|g(z)|:\,z\in \B\}>0\}$. If $N>n$, then  $\T_g^N(f)=Gf-g^\perp\K_0^N(\Omega f)$ is a bounded operator from $H^2(\theta)$ in $H^2(\theta)\times H^2(\theta)$ for any $\theta\in\A_2$.
\end{thm}

\section{The $H^2(\theta)$-corona problem for $m$ generators.} \label{sect:coronam}

 It is a well known fact that one way to prove the corona problem with $m$ generators is based in a successive resolution of several $\bar\p$ problems and a useful reformulation of the problem can be obtained by means of the so called the Koszul complex. We will use the formula  in Theorem 3.1 in \cite{An-Car3}, which gives in a sintetic way this composition. However, instead of obtaining properties of the solutions of each operator involved in such expression, we will rather obtain an estimate of the operator that solves the corona problem in each $H^2(\theta)$ for any weight $\theta\in\A_2$. The extrapolation theorem we have already cited in the introduction allows to deduce the general case for any $H^p(\theta)$ and any $\theta$ in $\A_p$. 

\subsection{The Koszul complex} 
Let $E=\{e_1,...,e_m\}$ be a basis in $\C^m$ and let $E^*$ be the corresponding dual basis.
We denote by $\Lambda^l=\Lambda^l(E)$ the elements $e_I=e_{i_1}\sqcap \ldots \sqcap e_{i_l}$ where $I=\{i_1,\ldots,i_l\}$, of degree $l$ of the exterior algebra  $\Lambda=\Lambda(E)$.  In order to avoid confusions, we use $\sqcap$ to denote the exterior multiplication in $\Lambda$ and $\land$ to denote the exterior product of differential forms. If $v^*\in E^*$,  $\delta_{v^*}:\Lambda^{l+1}\rightarrow \Lambda^l$ denotes the anti-derivation  defined by 
$$
\delta_{v^*}(e_{i_1}\sqcap\ldots \sqcap e_{i_l})=\sum_{j=1}^l(-1)^{j-1}v_j 
e_{i_1}\sqcap\ldots \sqcap e_{i_{j-l}}\sqcap e_{i_{j+l}}\ldots \sqcap e_{i_l}. 
$$

Let $\mathcal{E}_q$ denote the space of $(0,q)$-forms with coefficients in $\mathcal{C}^\infty(\bar \B)$ and $\mathcal{E}=\cup_{q=0}^n\mathcal{E}_q$.

We also consider the space $\mathcal{E}_q(\Lambda)$  of $\Lambda$ valued forms 
$$\sum_{I} \eta_I e_I ,\qquad \eta_I\in \mathcal{E}_q,$$
and the union of these spaces $\displaystyle{\mathcal{E}(\Lambda)=\cup_{q=0}^n\mathcal{E}_q(\Lambda)}$. 

We will use similar notations to consider other $\Lambda$-valued spaces of functions. For instance, $H^2(\theta,\Lambda)$ consists of  sums of  $h_I(z)e_I$ with $h_I\in H^2(\theta)$. 

For $F=\sum_{I} \eta_I e_I,\,H=\sum_{J} \vartheta_J e_J\in  \mathcal{E}(\Lambda)$, we let $$F\sqcap G=\sum_{I,J}\eta_I\land \vartheta_J\,e_I\sqcap e_J.$$

If $\mathcal{K}:\mathcal{E}\rightarrow \mathcal{E}$ is a linear operator, we will also use  $\mathcal{K}$ to denote the operator defined on $\mathcal{E}(\Lambda)$ by  $\mathcal{K}(\eta_I\,e_I)=\mathcal{K}(\eta_I)\,e_I$. If $h^*=\sum_{j=1}^m h_j(z)\sqcap e_j*\in \mathcal{E}_0((\Lambda^*)^1)$ (that is $h_j\in \mathcal{C}^\infty(\bar B)$), let $\delta_{h^*}(\eta_Ie_I) =\eta_I\sqcap\delta_{h^*}(e_I)=\sum_{j=1}^m h_j\eta_I\delta_{e_j^*}e_I$.
We denote by 
$\delta_{h^*}\mathcal{K}:\mathcal{E}(\Lambda)\rightarrow \mathcal{E}(\Lambda)$ the composition of $\mathcal{K}$ and $\delta_{h^*}$, that is  $(\delta_{h^*}\mathcal{K})(\eta_I e_I)=\mathcal{K}(\eta)\sqcap\delta_{h^*}(e_I)=
\sum_{j=1}^m h_j\mathcal{K}(\eta_I)\delta_{e_j^*}e_I$.

Using these notations, let us give an explicit formula to solve the corona problem. Observe that 
$\sum_{j=1}^m g_j F_j=f$ can be written as $\delta_{g^*} F=f$, where $g^*=\sum_{j=1}^m g_j(z) e_j^*$ and  $F=\sum_{j=1}^m F_j(z) e_j$.

Let us recall some integral operators $\K:\mathcal{E}\rightarrow \mathcal{E}$  satisfying $\bar\p \K(\eta)=\eta$ for any $(0,q+1)$-form satisfying $\bar\p \eta=0$. 

For $\,N\ge 0,\,$ consider the kernel
\begin{equation*}
\K^N(w,z)=\,\sum_{k=0}^{n-1}\,c_{k,N}\frac{(1-|w|^2)^{N+k}}{(1-\bar
w z)^{N+k}}\,\frac{( s\land (\bar\p\, s)^{n-1-k})(w,z)}{\phi^{n-k}(w,z)}\land
(\gamma(w))^k,
\end{equation*}
where $\bar\p=\bar\p_w+\bar\p_z$ ($\p$ in both variables $w$ and $z$), and 
\begin{equation}\label{eqn:sectKN} \begin{split}
\gamma(w)&=\bar\p\,\dfrac{\p|w|^2}{1-|w|^2}=\frac{\bar\p\p |w|^2}{1-|w|^2}+\frac{\bar\p|w|^2\land \p |w|^2}{(1-|w|^2)^2}\\
 s(w,z)&=(1- w \bar z)\p|w|^2 -(1-|w|^2)\p(w\bar z)\\
&=(1- w \bar z)\p(|w|^2 -w\bar z)+(|w|^2-w\bar z)\p(w\bar z)\\
\phi(w,z)&=|1-\bar w z|^2-(1-|w|^2)(1-|z|^2)\\
&=|(w-z)\bar w|^2+(1-|w|^2)|w-z|^2
\end{split}\end{equation}

Let $\K^N=\sum_{q=0}^{n-1}\K^N_q$ where $\K^N_q$ denotes the component in $\K^N$ of bidegree
$(0,q)$ in $z$ and $(n,n-q-1)$ in $w$. If $q=0$, then  
$\K_0^N(w,z)$  coincides with the kernel in Proposition \ref{prop:KN1}.

Formulas \eqref{eqn:sectKN} together with
\begin{align*}
\bar\p_w s(w,z)&=(1-w\bar z)\bar\p_w\p_w|w|^2-\p_w(w\bar z)\land\bar\p_w|w|^2\\
\bar\p_z s(w,z))&=-\bar\p_z(w\bar z)\land\p_w|w|^2-(1-|w|^2)\bar\p_z\p_w(w\bar z)\\
&=\bar\p_z(|z|^2-w\bar z)\land\p_w|w|^2-(1-|w|^2)\bar\p_z\p_w(w\bar z)\\
&\qquad -\bar\p_z |z|^2\land\p_w|w|^2,\\
\bar\p_w (w\bar z)\land\p_w|w|^2
&=\p_w (w\bar z-|w|^2)\land\p_w|w|^2,
\end{align*}  
give (see p. 69 \cite{Cu}) the following decomposition of $\K_q^N(w,z)$.
\begin{lem} \label{lem:decomK}
\begin{equation}\label{eqn:decomK}
\K_q^N(w,z)=\K_q^{N,1}(w,z)+\K_q^{N,2}(w,z)\land \bar\p|w|^2+\K_q^{N,3}(w,z)\land \bar\p|z|^2,
\end{equation}
with the folowing estimates:
 \begin{equation}\label{eqn:estKN3}\begin{split}
 |\K_q^{N,1}(w,z)|&\lesssim \L^{N+1}_{N+1-n,n-1/2}(w,z)\\
 |\K_q^{N,2}(w,z)|,\,|\K_q^{N,3}(w,z)|&\lesssim \L^{N+1/2}_{N+1-n,n-1/2}(w,z)
 \end{split}\end{equation}
  (we recall that the definition of the operators $\L^{N}_{M,L}$ is given in Definition \ref{def:operators})
\end{lem}

Note that, if $q=0$, then $\K_0^N$ does not contain  terms $d\bar z_j$, and therefore $\K^{N,3}_0=0$. Analogously, $\K^{N,2}_{n-1}=0$.

If $\z\in\bB$, then $\phi(w,\z)=|1-\z\bar w|^2$ and 
\begin{equation}\label{eqn:estKN3bB}
|\K_0^{N,1}(w,\z)|\lesssim \L^{N+1}_{n+N,0}(w,\z),\qquad |\K_0^{N,2}(w,\z)|\lesssim \L^{N+1/2}_{n+N,0}(w,\z).
\end{equation}

Observe that by \eqref{eqn:estKN3}, $|\K_q^{N,1}|$ is bounded by a kernel of type 1, and that $|\K_q^{N,2}|$ and $|\K_q^{N,3}|$ are bounded by  kernels of type 1/2. Therefore, $|\K_q^{N}|$ is globally  bounded by a kernel of type 1/2.

\vspace{.2cm}
Now, given $g=(g_1,\ldots, g_m)\in H^\infty$, satisfying  $\inf_{z\in \B}|g(z)|>0$, let $G_j=\dfrac{\bar g_j}{|g|^2}$ and let $G=\sum_{j=1}^m G_j(z) e_j$. Clearly, $\delta_{g^*} (G)=g\,G=1$. 

Then, we will use the following formula which provides solutions of the corona problem on Hardy spaces.

\begin{thm}[\cite{An-Car3}] \label{thm:Koszul} If  $g=(g_1,\ldots, g_m)\in H^\infty$ satisfies  $\inf_{z\in \B}|g(z)|>0$, then the linear operator 
\begin{equation}\label{eqn:Koszul}
\T_g^N(f)=\sum_{k=0}^{\min(n,m-1)} (-1)^k\left(\delta_{g^*}\K^N\right)^k \left(fG\sqcap (\bar \p G)^k\right),
\end{equation}
maps $H^p$ to  $H^p(\Lambda^1)$,\,$1\le p<\infty$, and $\delta_{g^*}(\T_g^N(f))=f$.
\end{thm}

 In order to facilitate the reading of this paper, we will give the explicit computations of $\T_g^N(f)$ for $m=2$ and $m=3$, and $n\ge 3$. 

If $m=2$, then formula \eqref{eqn:Koszul} coincides with the one of Section \ref{sect:corona2}. In order to prove this, observe that by bidegree reasons, the term $k=1$ in \eqref{eqn:Koszul} is 
\begin{align*} &(\delta_{g^*}\K^N)\left(f(G_1 e_1+G_2 e_2)\sqcap(\bar\p G_1 e_1+\bar\p G_2 e_2)\right) \\
&\qquad=(\delta_{g^*}\K^N_0)\left(f(G_1 e_1+G_2 e_2)\sqcap(\bar\p G_1 e_1+\bar\p G_2 e_2)\right)
\end{align*}
and that 
\begin{align*} 
(\delta_{g^*}\K^N_0)&\left(f(G_1 e_1+G_2 e_2)\sqcap(\bar\p G_1 e_1+\bar\p G_2 e_2)\right)\\
&=(\delta_{g^*}\K^N_0)\left(f(G_1\bar\p G_2 -G_2\bar\p G_1)e_1\sqcap e_2\right)\\
&=\K^N_0(fG_1\bar\p G_2 - fG_2\bar\p G_1)(g_1 e_2- g_2 e_1)
\end{align*}

Following the notations of Section \ref{sect:corona2}, by \eqref{eqn:omega} we have $fG_1\bar\p G_2 - fG_2\bar\p G_1=f\Omega$ and  therefore
$$\T_g^N(f)=(fG_1+g_2\K^N_0(f\Omega)) e_1+(fG_2-g_1\K^N_0(f\Omega)) e_2,$$
which coincides with \eqref{eqn:coronaT2}.

If $m=3$, then similar computations prove that  the term $k=1$ in \eqref{eqn:Koszul} is 
\begin{align*}
(\delta_{g^*}\K^N_0)(G\sqcap\bar\p G)&
=\K^N_0(fG_1\bar\p G_2 - fG_2\bar\p G_1)(g_1 e_2- g_2 e_1)\\
&+\K^N_0(fG_2\bar\p G_3 - fG_3\bar\p G_2)(g_2 e_3- g_3 e_2)\\
&+\K^N_0(fG_3\bar\p G_1 - fG_1\bar\p G_3)(g_3 e_1- g_1 e_3)
\end{align*}
Now, if $\Omega_{i,j}=\left|\begin{matrix} G_i&G_j\\\bar\p G_i&\bar\p G_j\end{matrix}\right|=G_i\bar\p G_j-G_j\bar\p G_i$, then 
\begin{equation} \label{eqn:deltaK2}\begin{split}
(\delta_{g^*}\K^N_0)(G\sqcap\bar\p G)&=(g_2\K^N_0)(f\Omega_{2,1})+g_3\K^N_0(f\Omega_{3,1}))e_1\\
&+(g_1\K^N_0(f\Omega_{1,2})+g_3\K^N_0(f\Omega_{3,2}))e_2\\
&+(g_1\K^N_0(f\Omega_{1,3})+g_2\K^N_0(f\Omega_{2,3}))e_3.
\end{split}\end{equation}

In order to calculate the term $k=2$ in \eqref{eqn:Koszul},
let 
\begin{align*}
\Omega_{123}&=\left|\begin{matrix} G_1&G_2&G_3\\\bar\p G_1&\bar\p G_2&\bar\p G_3\\\bar\p G_1&\bar\p G_2&\bar\p G_3
\end{matrix}\right|\\
&=2\left(G_1\bar\p G_2\land\bar\p G_3
+G_2\bar\p G_3\land\bar\p G_1+G_3\bar\p G_1\land\bar\p G_2\right).
\end{align*}
It is easy to check  that  $G\sqcap\bar\p G\sqcap\bar\p G=\Omega_{123}\,\, e_1\sqcap e_2\sqcap e_3$.
The use of the determinants of forms to formulate the Koszul complex can be found in \cite{Lin}.  

Therefore, 
\begin{align*}
&(\delta_{g^*}\K^N_1)(\Omega_{123}\,\, e_1\sqcap e_2\sqcap e_3)\\
&=g_1\K^N_1(\Omega_{123})\,e_2\sqcap e_3+ g_2\K^N_1(\Omega_{123})\,e_3\sqcap e_1+g_3\K^N_1(\Omega_{123})\,e_1\sqcap e_2,
\end{align*}
and
\begin{align*}
(\delta_{g^*}\K^N_0)&(\delta_{g^*}\K^N_1)(\Omega_{123}\,\, e_1\sqcap e_2\sqcap e_3)\\
&=\K_0^N(g_1\K^N_1(\Omega_{123}))\,(g_2 e_3- g_3 e_2)\\
&+ \K_0^N(g_2\K^N_1(\Omega_{123}))\,(g_3e_1-g_1 e_3)\\
&+\K_0^N(g_3\K^N_1(\Omega_{123}))\,(g_1 e_2- g_2e_1).
\end{align*}

Observe that in general we have
\begin{lem}\label{lem:coefficients1}
The coefficients of $\left(\delta_{g^*}\K\right)^k \left(fG\sqcap (\bar \p G)^k\right)$, $k\ge 1$ are linear combinations of terms of type
\begin{equation}\label{eqn:coefficients1}
F_le_l=g_{i_0}\K_0(g_{i_1}(\K_1(....(g_{i_{k-1}} \K_{k-1}(f G_{j_0} \bar\p G_{j_1}\land\ldots\land \bar\p G_{j_k}))))))\, e_l.
\end{equation}
\end{lem}
\vspace{.2cm}

To conclude, for completeness, we recall the proof of the fact that $\T_g^N(f)\in H(\Lambda^1)$ given in Theorem 3.1 in \cite{An-Car3}.

Let $r=\min\{n,m-1\}$ and let $V_k=G\sqcap(\bar\p G)^k$. We define by induction the forms $U_r=V_r$ and $U_k=V_k-(\delta_{g^*}\K)(U_{k+1})$, $0\le k<r$.
Observe that $U_0=\T_g^N(f)$, $\delta_{g^*}(\bar\p G)=(\bar\p(\delta_{g^*} G))=0$ and  $\delta_{g^*}(V_k)=(\bar\p G)^k$. 

We want to prove that $\bar\p U_k=0$ for all $0\le k\le r$. If $r=n$, then 
by bidegree reasons, the form $V_r=G\sqcap(\bar\p G)^r$ satisfies $\bar\p V_r=0$. If $r=m-1$, then using $\delta_{g^*}\bar\p G=0$ we also obtain $\bar\p V_r=0$.
Assume that $\bar\p U_{k+1}=0$. Since $\delta_{g^*}^2=0$, we have $\delta_{g^*}U_{k+1}=\delta_{g^*}V_{k+1}=(\bar\p G)^{k+1}=\bar\p V_{k}$. Therefore,  we have  $\bar\p((\delta_{g^*}\K_k)(U_{k+1})) =\delta_{g^*}(\bar\p\K_k(U_{k+1}))=\delta_{g^*}U_{k+1}=\bar\p V_{k}$, which proves that $\bar\p U_k=0$.

\subsection{Estimates of $F_l$ in \eqref{eqn:coefficients1}}
We want to prove that if $f\in H^2(\theta)$, then the terms $F_l$ are in $L^2(\theta)$. This result will be a consequence of the same technique used to prove the case of two generators, which permit us to estimate the cases $k=0$ and $k=1$, and the following proposition. For $k\geq 2$ we will use other arguments, since in this case by Lemma \ref{lem:decompforms} below it is not necessary to use the decomposition of the operator given in (i) of Proposition \ref{prop:KN1}.

\begin{prop}\label{prop:estcoef} For $N$ large enough and $k\ge 2$, we have
\begin{align*}
&\left|(\delta_{g^*}\K^N)^k \left(fG\sqcap (\bar \p G)^k\right)(\z)\right|\\
&\qquad\lesssim \int_{\B}\frac{|f(w)|((1-|w|^2)|\p g(w)|^2
+ |\p_T g(w)|^2)\,(1-|w|^2)^n}{|1-\z\bar w |^{2n}}d\nu(w).
\end{align*}
\end{prop}

Assuming this result, it is easy to prove the corona theorem for $p=2$.

\begin{thm} \label{thm:corH2} Let $g=(g_1,\ldots,g_m),\,g_j\in H^\infty$ satisfying  $\inf_{z\in \B}|g(z)|>0$. If $N$ is large enough, then  the operator $\T^N_g$ in Theorem \ref{thm:Koszul} maps $H^2(\theta) $ to $H^2(\theta,\Lambda^1)$ for any $\theta\in\A_2$.
\end{thm}

\begin{proof} The estimate of $(\delta_{g^*}\K^N)^0(fG)=fG$, corresponding to the term $k=0$ in \eqref{eqn:Koszul}, is clear. The estimate of the term $k=1$, that is  $(\delta_{g^*}\K^N)(fG\sqcap\bar\p G)$, follows arguing as in the case  of two generators (observe that arguing as in \eqref{eqn:deltaK2} the coefficients of the terms that appear in the representation of $(\delta_{g^*}\K^N)(fG\sqcap\bar\p G)$ are of the same type of the expressions $g_j\K^N_0(f\Omega)$ considered in Section 5).

Therefore, it remains to consider the terms $k\ge 2$. 
For any $\psi\in L^2(\theta^{-1})$, Proposition \ref{prop:estcoef} gives
\begin{equation*}
\begin{split}
&\left|\int_{\bB} F_l(\z)\psi(\z)d\sigma(\z)\right|\lesssim\int_{\B} |f(w)||\p_T g(w)|^{2}\mathbb{P}(|\psi|)(w)d\nu(w)\\
&\qquad \qquad\qquad+ \int_{\B} (1-|w|^2)|f(w)||\p  g(w)|^{2}\mathbb{P}(|\psi|)(w)d\nu(w).
\end{split}\end{equation*} 

Thus, arguing as in Proposition \ref{prop:corona2}, 
 H\"older's inequality, the fact that $\Theta^\frac 12 {\Theta'}^\frac 12\approx 1$ and Proposition \ref{lem:estKw12} give
\begin{align*}
\left|\int_{\bB} F_l(\z)\psi(\z)d\sigma(\z)\right|
&\lesssim \left(\int_{\bB} |f|^2d\mu_{g,\theta}\right)^{1/2}
\left(\int_{\bB} |\mathbb{P}(\psi)|^2d\mu_{g,\theta'}\right)^{1/2}\\
&\lesssim \|f\|_{H^2(\theta)}\|\psi\|_{L^2(\theta')}
\end{align*}
which proves that $F_l(\z)\in L^2(\theta)$ and $\|F_l\|_{L^2(\theta)}\lesssim \|f\|_{H^2(\theta)}$.
\end{proof}

\vspace{.3cm}

Therefore, it remains to prove Proposition \ref{prop:estcoef}.

\begin{lem}\label{lem:decompforms} For $k\ge 2$, 
\begin{equation}
\label{eqn:form1}  (G_{j_0} \bar\p G_{j_1}\land\ldots\land \bar\p G_{j_k})(w)=\tilde G_R(w)+\tilde G_T(w)\land \bar\p|w|^2
\end{equation}
with 
\begin{equation}\label{eqn:form11} \begin{split}
|\tilde G_R(w)|&\lesssim (1-|w|^2)^2|\p g(w)|^2+(1-|w|^2)^{1-k/2}|\p_T g(w)|^{2}\\
 |\tilde G_T(w)|&\lesssim (1-|w|^2)^{1/2-k/2} \left[(1-|w|^2)|\p g(w)|^{2}+|\p_T g(w)|^2\right],
 \end{split}\end{equation}
 and consequently
 $$ |\tilde G(w)|\lesssim (1-|w|^2)^{1/2-k/2} \left[(1-|w|^2)|\p g(w)|^{2}+|\p_T g(w)|^2\right].$$
\end{lem}

\begin{proof}
The decomposition
\begin{align*}
\bar\p G_l(w)&=(1-|w|^2)\bar\p G_l(w)+\sum_{j=1}^n\sum_{i=1}^n\bar w_i w_j\bar D_i G_l(w) d\bar w_j\\
&\quad+\sum_{j=1}^n\sum_{i=1}^n\bar w_i(w_i\bar D_jG_l(w)-w_j\bar D_iG_l(w)) d\bar w_j\\
&=(1-|w|^2)\bar\p G_l(w)+\bar\Ra G_l(w) \bar\p |w|^2+\sum_{i,j} \bar w_i \bar D_{i,j}G_l(w)d\bar w_j,
\end{align*}
and $(1-|w|^2)^{1/2}|\p_T g(w)|+(1-|w|^2)|\p g(w)|\lesssim 1$, prove \eqref{eqn:form1}
with 
\begin{equation*}\label{eqn:form111} \begin{split}
|\tilde G_R(w)|&\lesssim \left[(1-|w|^2)|\p g(w)|+ |\p_T g(w)|\right]^{k}\\
&\lesssim (1-|w|^2)^2|\p g(w)|^2+(1-|w|^2)^{1-k/2}|\p_T g(w)|^{2}\\
 |\tilde G_T(w)|&\lesssim \left[(1-|w|^2)|\p g(w)|+ |\p_T g(w)|\right]^{k-1}|\p g(w)|
\end{split}\end{equation*}
Since $k\ge 2$, $(1-|w|^2)^{k-1}|\p g(w)|^k\lesssim (1-|w|^2)|\p g(w)|^2$ and 
\begin{align*}
|\p_T g(w)|^{k-1}|\p g(w)|&\lesssim (1-|w|^2)^{1-k/2}|\p_T g(w)||\p g(w)|\\
&\lesssim (1-|w|^2)^{1/2-k/2}\left(|\p_T g(w)|^2+(1-|w|^2)|\p g(w)|^2\right)
\end{align*}
which ends the proof.
\end{proof}

The next lemma is well-known (see for instance Lemma 2.5 in \cite{Or-Fa1}).

\begin{lem}\label{lem:estbiP} If $0\le A,B<N<n+A+B$ , then 
$$
\int_{\B}\frac {(1-|u|^2)^{N-1}}{|1-z\bar u|^{n+A} |1-u\bar w|^{n+B}} \,d\nu(u)\lesssim \frac{1}{|1-\z\bar w|^{n+A+B-N}}.
$$
\end{lem}

\begin{lem} \label{lem:estcomposK} If the kernel $\L^{N}_{M,L}$ has type $\kappa=n+N-M-2L>0$, then for $\kappa-n<A\le N$ and $B\ge 0$,
$$\int_{\B}\L^N_{n+A,0}(z,\z) \L^{N+B}_{M+B,L}(w,z)d\nu(z)\lesssim \L^N_{n+A-\kappa,0}(w,\z) .$$

 Observe that $type(\L^N_{n+A-\kappa,0})=type(\L^N_{n+A,0})+type(\L^{N}_{M,L})$.
\end{lem}

\begin{proof} Since $\L^{N+B}_{M+B,L}(w,z)\lesssim \L^{N}_{M,L}(w,z)$ we can consider $B=0$. 

The left hand side term in the above inequality is
$$I(w,\z)=(1-|w|^2)^{N-1}
\int_{\B}\frac{(1-|z|^2)^{N-1}}{|1-\z\bar z|^{n+A}|1-z\bar w|^M\,\phi(w,z)^{L}}d\nu(z) .$$

Let $\varphi_w(z)$ denotes the automorphism of the unit ball which maps $w$ to 0. We will use the change $u=\varphi_w(z)$ and the formulas in Section 2.2 in \cite{Ru} to  reduce the above estimate to the one of Lemma \ref{lem:estbiP}.

Since 
\begin{equation}\label{eqn:Rudin} 
1-\varphi_w(z)\overline{\varphi_w(u)}=
\frac{(1-|w|^2)(1-z\bar u)}{(1-z\bar w)(1-w\bar u)},
\end{equation}
we have
\begin{align*}
1-|\varphi_w(u)|^2&=\dfrac{(1-|w|^2)(1-|u|^2)}{|1-u\bar w|^2},\quad\text{   and }\\
\phi(w,z)&=|1-z\bar w|^2|\varphi_w(z)|^2.
\end{align*} 
Therefore, the change of variables $u=\varphi_w(z)$ gives 
\begin{align*}
&I(w,\z)\\
&=c\int_{\B} 
\frac{(1-|w|^2)^{N-1}(1-|\varphi_w(u)|^2)^{N-1}}{|1-\z\overline{\varphi_w(u)}|^{n+A}|1-\varphi_w(u)\bar w|^{M+2L}|u|^{2L}}\frac{(1-|w|^2)^{n+1}}{|1-u\bar w |^{2n+2}} \,d\nu(u).
\end{align*}

By \eqref{eqn:Rudin}, 
\begin{align*}
1-\varphi_w(u)\bar w&=1-\varphi_w(u)\overline{\varphi_{w}(0)}=\frac{1-|w|^2}{1-u\bar w}\\
1-\z\overline{\varphi_w(u)}&=1-\varphi_w(\varphi_w(\z))\overline{\varphi_w(u)}=\frac{(1-\z\bar w)(1-\varphi_w(\z)\bar u )}{1-u\bar w}
\end{align*}
and therefore
\begin{align*}
&I(w,\z)\\
&=c\frac{(1-|w|^2)^{N-1+\kappa}}{|1-\z\bar w|^{N+A}}
\int_{\B}\frac {(1-|u|^2)^{N-1}}{|1-\varphi_w(\z)\bar u|^{n+A}|1-u\bar w |^{N+\kappa-A}|u|^{2L}} \,d\nu(u)
\end{align*}

We decompose the above integral in the sum of the integral in the ball of radious 1/2 and of the integral in its complementary set. Since $L<2n$ and $|1-\bar u z|\approx 1$ on $\frac 12 \B=\{u\in\B; |u|\le 1/2\}$, the  integral in this set is bounded. By Lemma \ref{lem:estbiP}, the  integral in the complementary of $\frac 12 \B$ is bounded by $\dfrac{1}{|1-\varphi_w(\z)\bar w|^\kappa}=\dfrac{|1-\z\bar w|^\kappa}{(1-|w|^2)^\kappa}$, which concludes the estimate.
\end{proof}

\begin{proof}[Proof of Proposition \ref{prop:estcoef}:] 
Observe that by decomposition \eqref{eqn:decomK}, and the facts that $\bar\p|w|^2\land \bar\p|w|^2=0$, 
\begin{equation}\label{eqn:decomknl}\begin{split}
&\K_q^{N,2}(z,u)\land\K_{q+1}^{N,3}(w,z)=0,\,\,\text{ for all $0\le q\le n-1$,\,\, and}\\
& \K^{N,3}_0=0,
\end{split}\end{equation} 
the  term in \eqref{eqn:coefficients1} is a sum of terms of type 
$$F_1=g_{i_0}\K_0^{N,j_0}(g_{i_1}(\K_1^{N,j_1}(....(g_{i_{k-1}} \K_{k-1}^{N,j_{k-1}}(f\tilde G ))....))),$$
with $j_l=1,2,3$ and at last one of them equal to 1, and one term of type
$$F_2=g_{i_0}\K_0^{N,2}(g_{i_1}(\K_1^{N,2}(....(g_{i_{k-1}} \K_{k-1}^{N,2}(f\tilde G_T ))....))).$$

Observe that by \eqref{eqn:decomknl}, all terms including $\K_{q}^{N,3}(f\tilde G)$
are considered in the first type $F_1$.

Since  $|\K^{N,1}_q|$ is bounded by a kernel of type 1 and $|\K^{N,2}_q|$, $|\K^{N,3}_q|$ are bounded by a kernel of type 1/2, the kernels in  $F_1$ are bounded by a product of kernels of type 1 or 1/2 and whose sum of types are greater or equal to $(k-1)/2+1=(k+1)/2$.

Analogously, the kernels in  $F_2$ are bounded by a product of kernels of type  1/2 and whose sum of types is equal to $k/2$.

Therefore, if $N$ is large enough, then the pointiwise estimate of $\tilde G$ in Lemma \ref{lem:decompforms}, together Lemma \ref{lem:estcomposK} give
\begin{align*}
|F_1(\z)|&\lesssim \int_{\B}\frac{(1-|w|^2)^{N-1/2}}{|1-\z\bar w|^{n+N-k/2}}(1-|w|^2)^{1/2-k/2} d\mu_g(w)\\
 &=\int_{\B}\frac{(1-|w|^2)^{N-k/2}}{|1-\z\bar w|^{n+N-k/2}}d\mu_g(w),
\end{align*}
where $d\mu_g(w)=\left[(1-|w|^2)|\p g(w)|^2+|\p_T g(w)^2|\right]d\nu(w)$.

Analogously, 
\begin{align*}
|F_2(\z)|&\lesssim \int_{\B}\frac{(1-|w|^2)^{N-1/2}}{|1-\z\bar w|^{n+N+1/2-k/2}}(1-|w|^2)^{1-k/2}d\mu_g(w)\\
 &=\int_{\B}\frac{(1-|w|^2)^{N+1/2-k/2}}{|1-\z\bar w|^{n+N+1/2-k/2}}d\mu_g(w).
\end{align*}

Therefore, taking $N \ge n+k/2$ we obtain the estimate in Proposition \ref{prop:estcoef}.
\end{proof}

\section{End of the proof of Theorem \ref{thm:HpMorrey}}

\subsection{Corona theorem for weighted Hardy spaces}

In order to prove the corona theorem for $H^p(\theta)$, we will use the following extrapolation theorem proved in \cite{RuFa} (see also p.223 \cite{St}).

\begin{thm} \label{thm:extrap}
Let $1<r<+\infty$, and $\T$ a sublinear operator which is bounded on $L^r(\theta)$ for any $\theta\in \A_r$, with constant depending only on the constant 
$\A_r(\theta)$ of the condition $\A_r$. Then $\T$ is bounded on $L^p(\theta)$ for any $1<p<+\infty$ and any $\theta\in \A_p$, with constant depending only on
$\A_p(\theta)$.
\end{thm}

\begin{thm} \label{thm:Hp7} Let $1< p<\infty$ and $0<s<n/p.$ Let $g_1,\ldots, g_m\in H^\infty$. Then, the following assertions are equivalent:

\begin{enumerate}
\item \label{item:H1} The functions $g_k$, $k=1,\ldots,m$ satisfy  $\inf\{|g(z)|:\,z\in \B\}>0.$
\item \label{item:H2} $\M_g$ maps $H^p(\theta)\times \cdots\times H^p(\theta)$ onto $H^p(\theta)$ for any $1<p<\infty$ and any $\theta\in\A_p$.
\item \label{item:H3} $\M_g$ maps $H^p(\theta)\times \cdots\times H^p(\theta)$ onto $H^p(\theta)$ for some $1<p<\infty$ and some $\theta\in\A_p$.
\item \label{item:H6} $\M_g$ maps $H^2(\theta)\times \cdots\times H^2(\theta)$ onto $H^2(\theta)$ for any   $\theta\in\A_2$.
\end{enumerate}
\end{thm}

\begin{proof} We will follow the  scheme:
\eqref{item:H2}$\Rightarrow$ \eqref{item:H3}$\Rightarrow$ \eqref{item:H1}$\Rightarrow$ \eqref{item:H3}$\Rightarrow$ \eqref{item:H2}

Clearly \eqref{item:H2} $\Rightarrow$ \eqref{item:H3}. The implication \eqref{item:H3} $\Rightarrow$ \eqref{item:H1} is proved in Proposition \ref{prop:necH}. The proof of  \eqref{item:H1} $\Rightarrow$ \eqref{item:H6} is given in Theorem \ref{thm:corH2} using the linear operator $\T_g^N$. The proof of \eqref{item:H6} $\Rightarrow$ \eqref{item:H2} follows from Theorem \ref{thm:extrap} applied to 
$r=2$ and to each one of the operators $\T=\T_{g,i}^N\circ \mathcal{C}$, $i=1,\dots,m$. Here $\T_{g,i}^N$ are the components of the operator $\T_g^N$ and $\mathcal{C}$ is the Cauchy kernel.
\end{proof}

\subsection{Corona theorem for Morrey spaces}

The following result was proved in Theorem 3.1 in \cite{Ar-Mi}.

\begin{thm} \label{thm:morreyA1} 
Let $\varphi$ and $\psi$ be nonnegative Borel measurable functions on $\bB$. Suppose that for each $\alpha\ge 1$ and every bounded weight $\theta\in\A_1$, such that $\A_1(\theta)\le \alpha$, there exists $c(\alpha)$ 
such that 
$$\int_{\bB} \varphi\theta d\sigma \le c(\alpha)\int_{\bB} \psi\theta d\sigma.$$

Then, for $0<t<n$, there exists a constant $C$ depending on $n$ and $t$, such that $\|\varphi\|_{M^{1,t}}\le C \|\psi\|_{M^{1,t}}$ for any $\varphi,\psi\in M^{1,t}$.
\end{thm}

\begin{thm} \label{thm:Morrey7} Let $1< p<\infty$ and $0<s<n/p.$ Let $g_1,\ldots, g_m\in H^\infty$. Then, the following assertions are equivalent:

\begin{enumerate}
\item \label{item:M1} The functions $g_k$, $k=1,\ldots,m$ satisfy  $\inf\{|g(z)|:\,z\in \B\}>0.$
\item \label{item:M4} $\M_g$ maps $HM^{p,s}\times\cdots\times  HM^{p,s}$ onto $HM^{p,s}$ for any $1<p<\infty$ and any $0<s<n/p$.
\item \label{item:M5} $\M_g$ maps $HM^{p,s}\times\cdots\times  HM^{p,s}$ onto $HM^{p,s}$ for some  $1<p<\infty$ and some  $0<s<n/p$.
\end{enumerate}
\end{thm}

\begin{proof}
The scheme of the proof of the Morrey case is similar and we will show in this case that
\begin{center}
\eqref{item:M4} $\Rightarrow$ \eqref{item:M5} $\Rightarrow$ \eqref{item:M1} $\Leftrightarrow$ \eqref{item:H2}[Theorem \ref{thm:Hp7}]  $\Rightarrow $ \eqref{item:M4}.
\end{center}

The first implication is obvious, and the proof of the second is given in Proposition \ref{prop:necM}. 
The proof of  \eqref{item:H2}[Theorem \ref{thm:Hp7}] $\Rightarrow $ \eqref{item:M4} follows from Theorem \ref{thm:morreyA1}. 
Observe that, if $1<p<\infty$,  $\varphi=|\T_g(f)|^p$, $\psi=|f|^p$ and $t=sp<n$, then the fact that $\A_1\subset \A_p$, \eqref{item:HM2}[Theorem \ref{thm:Hp7}] and 
Theorem \ref{thm:morreyA1} give 
$$\||\T_g(f)|\|_{M^{p,s}}^p=\||\T_g(f)|^p\|_{M^{1,sp}}\le C \||f|^p\|_{M^{1,sp}}=\|f\|_{M^{p,s}}^p,$$
which proves the result.

\end{proof}

\end{document}